\documentclass[12pt]{amsart}
\usepackage{amssymb,enumerate,mathrsfs, amsmath, amsthm}
\usepackage{url,titletoc,enumitem}
\usepackage{tikz-cd}

\DeclareMathOperator{\Gal}{Gal}

\newtheorem{theorem}{Theorem}[section]
\newtheorem{cor}[theorem]{Corollary}
\newtheorem{lemma}[theorem]{Lemma}

\newtheorem{prop}[theorem]{Proposition}
\newtheorem{definition}[theorem]{Definition}

\newtheorem{remark}[theorem]{Remark}

\newtheorem{conj}[theorem]{Conjecture}

\newtheorem{question}[theorem]{Question}

\newtheorem*{lemma*}{Lemma}
\newtheorem{Heuristic}[theorem]{Heuristic}

\numberwithin{equation}{section}

\def\bbR{ {\mathbb R}}

\def\bbF{ {\mathbb F}}
\def\bbQ{ {\mathbb Q}}
\def\bbZ{ {\mathbb Z}}

\def\cO{ {\mathcal O} }

\usepackage{biblatex}
\makeindex

\title{LINEAR RELATIONS AMONG GALOIS CONJUGATES OVER $\mathbb{F}_q(t)$}
\author{Will Hardt and John Yin}
\begin{document}

\maketitle 

\section{Abstract}
We classify the coefficients $(a_1,...,a_n) \in \mathbb{F}_q[t]^n$ that can appear in a linear relation $\sum_{i=1}^n a_i \gamma_i =0$ among Galois conjugates $\gamma_i \in \overline{\mathbb{F}_q(t)}$. We call such an $n$-tuple a \textit{Smyth tuple}. Our main theorem gives an affirmative answer to a function field analogue of a 1986 conjecture of Smyth \cite{Smy86} over $\mathbb{Q}$. Smyth showed that certain local conditions on the $a_i$ are necessary and conjectured that they are sufficient. Our main result is that the analogous conditions are necessary and sufficient over $\mathbb{F}_q(t)$, which we show using a combinatorial characterization of Smyth tuples from \cite{Smy86}. We also formulate a generalization of Smyth's Conjecture in an arbitrary number field that is not a straightforward generalization of the conjecture over $\mathbb{Q}$ due to a subtlety occurring at the archimedean places.

\section{Introduction}

The question of how prevalent linear relations among Galois conjugates are has been studied from multiple angles.  In \cite{Ber04} it is shown that for a Hilbertian field $K$ and all but finitely many nonnegative integers $n$, there exists an algebraic number $\alpha \in \overline{K}$ of degree $2^n n!$ whose conjugates span a vector space of dimension $n$. In these cases, the dimension of relations between conjugates is $2^n n!-n$, and so in this sense, linear relations among conjugates are plentiful.

In an older paper, Smyth showed \cite[Cor. 2]{Smy86} that for any $a_1, \dots, a_n \in \bbZ$ with $gcd(a_1, \dots, a_n)=1$, if there exist Galois conjugates $\gamma_1, \dots, \gamma_n \in \overline{\mathbb{Q}}$ such that $\sum_{i=1}^n a_i \gamma_i = 0$, then the $a_i$ necessarily satisfy the following two properties. \begin{enumerate}
    \item $|a_i| \leq \sum_{j \neq i} |a_j|$ for all $i$
    \item Every prime $p$ divides at most $n-2$ of the $a_i$.
\end{enumerate}

Smyth conjectured that the converse is true as well over $\mathbb{Q}$.

\begin{conj} [Smyth's Conjecture] \label{Smyth-Conj} \cite{Smy86} If $(a_1, \dots, a_n) \in \mathbb{Q}^n$ satisfy (1) and (2), then there exist Galois conjugates $\gamma_1, \dots, \gamma_n$ so that $\sum_{i=1}^n a_i \gamma_i = 0$.
\end{conj}

A natural way to generalize (1) and (2) to an arbitrary field $K$ is as follows.
\begin{itemize}
     \item[(1')] For any archimedean absolute value $| \cdot |$ of $K$, we have $|a_i| \leq \sum_{j \neq i} |a_j|$ for all $i$.
    \item[(2')] For any nonarchimedean absolute value $| \cdot |$ of $K$, we have $|a_i| \leq \max_{j \neq i} |a_j|$ for all $i$.
\end{itemize}
We call (1') and (2') the \textit{absolute value criteria over $K$}.


Note that for any $a \in K^*$, a tuple $(a_1,...,a_n) \in K^n$ satisfies the absolute value criteria if and only if $(a a_1,...,a a_n)$ does, and similarly, $(a_1,...,a_n)$ is a Smyth tuple if and only if $(aa_1,...,aa_n)$ is a Smyth tuple. Hence, when $K$ is the field of fractions of a principal ideal domain, it is enough to look at coprime tuples $(a_1,...,a_n)$.

Conjecture \ref{Smyth-Conj} remains open. Our main result is Theorem \ref{smyth-conj-Fqt}, which answers in the affirmative the natural analogue of Smyth's question over $\bbF_q(t)$. We will also formulate in Conjecture \ref{Smyth-Conj-Num-Field} a generalization of Conjecture \ref{Smyth-Conj} for arbitrary number fields, which involves a subtlety not present in the cases of $\mathbb{Q}$ or $\mathbb{F}_q(t)$. A MathOverflow post of David Speyer \cite{MO264035} proves a special case of this conjecture.




In Section \ref{section-prelims} we record some terminology that we will use throughout the paper. In Section \ref{section-Fqt} we set up Theorem \ref{smyth-conj-Fqt}, which we go on to prove in Section \ref{proof section}. Section \ref{deg-of-gal-conj} examines how close the Galois relations constructed in our proof of Theorem \ref{smyth-conj-Fqt} are to ``as small as possible.'' In Section \ref{gal-grp} we show that the conjugates in our constructed linear relations for Smyth triples can be chosen to have the full symmetric group as their Galois group and in Section \ref{section -- heuristic} we give heuristic reasoning for why we should expect the Galois group of linear Galois relations to be large in general. In Section \ref{smyth-over-Q}, we record some results from Section \ref{deg-of-gal-conj} that also apply over $\mathbb{Q}$. Finally, in Section \ref{num-fields}, we formulate a number field generalization of Conjecture \ref{Smyth-Conj}, and then reduce it slightly to Conjecture \ref{reduced-conj}.

\section{Preliminaries} \label{section-prelims}
We lay out some terminology and a background result.

We say that $\gamma_1,...,\gamma_n$ are \textit{Galois conjugates over K} if they are all roots of the same irreducible polynomial over $K$. (When the base field $K$ is clear from the context, we will often omit it from our terminology.) Throughout the paper, we assume Galois conjugates are nonzero.

A \textit{linear Galois relation} will mean a linear relation $\sum_{i=1}^n a_i \gamma_i = 0$ among Galois conjugates $\gamma_1,...,\gamma_n \in \overline{K}$ with coefficients $a_i \in K$.

We will call a tuple $(a_1,...,a_n) \in \mathbb{F}_q[t]^n$ \textit{coprime} if $a_1,...,a_n$ generate the unit ideal (and similarly for $(a_1,...,a_n) \in \mathbb{Z}^n$). The (logarithmic) \textit{height} of a coprime tuple $(a_1,...,a_n) \in \mathbb{F}_q[t]^n$ is $\max_i \deg(a_i)$. The (logarithmic) \textit{height} of a coprime tuple $(a_1,...,a_n) \in \mathbb{Z}^n$ is $\max_i \log|a_i|$. A coprime tuple is said to be a \textit{Smyth tuple} if its coordinates appear as the coefficients of a linear Galois relation.

A \textit{balanced multiset of tuples} -- which we will sometimes shorten to ``\textit{balanced multiset}'' -- with respect to a tuple $(a_1,...,a_n)$ is a nonempty collection of nonzero solutions $\{(x_{i_1},...,x_{i_n}) \in K^n\}_{i=1}^N$ to the equation $\sum_{i=1}^n a_i x_i = 0$ such that the multiset $\{x_{i_j}\}_{i=1}^N$ is independent of $j$; in other words, for all $x \in K$, $x$ appears in each of the $N$ coordinate positions the same number of times.

A balanced multiset $\{(x_{i_1},...,x_{i_n}) \in K^n\}_{i=1}^N$ is a \textit{1-factor} if the multiset $\{x_{i_1}\}_{i=1}^N$ is in fact a set. (To see the justification for this terminology, consider the hypergraph whose vertices are elements of $K$ and whose edges are ordered tuples $(x_1,...,x_n)$ of solutions to $\sum_{i=1}^n a_i x_i = 0$. Then a 1-factor in this hypergraph is precisely our definition of 1-factor.)

The \textit{degree of a linear Galois relation} $\sum_{i=1}^n a_i \gamma_i = 0$ will refer to the degree $[K(\gamma_1): K]$ of the conjugates in the linear relation. Similarly the \textit{Galois group of a linear Galois relation} $\sum_{i=1}^n a_i \gamma_i = 0$ will mean $\Gal(E/K)$, where $E$ is the Galois closure of $K(\gamma_1)$.

Following standard terminology, we will call a Galois extension $L/K$ a $G$-extension if $\Gal(L/K) \cong G$. 

Our final preliminary is the fact that over any Hilbertian field, there exist $S_d$-extensions for all $d \geq 1$. This is a special case of \cite[Prop. 8]{Ber04}. We will apply this fact to the Hilbertian field $\mathbb{F}_q(t)$.



\section{Smyth's Conjecture over $\bbF_q(t)$} \label{section-Fqt}

Although our primary focus in this section (and this paper) is the function field $\mathbb{F}_q(t)$, we will often state definitions and theorems over an arbitrarily field $K$ when it is not substantially more complicated to do so.


Our approach in this paper benefits from having several equivalent notions of Smyth tuples. Most importantly, Smyth proved a combinatorial characterization of Smyth tuples involving balanced multisets. We record this, along with one other characterization of Smyth tuples, in Proposition \ref{equivalent-conditions}.

\begin{prop} \label{equivalent-conditions}
Let $K$ be any field such that there exist $S_d$-extensions of $K$ for all $d \geq 1$. The following are equivalent for $(a_1,...,a_n) \in K^n$.
\begin{enumerate}
    \item $(a_1,...,a_n)$ is a Smyth tuple
    \item There exists a balanced multiset of tuples with respect to $(a_1,...,a_n)$
    \item There exist permutation matrices $X_1,...,X_n$ such that $\det(\sum_{i=1}^n a_i X_i) = 0$
\end{enumerate}
\end{prop}
\begin{proof}
Smyth proved (1) $\iff$ (2) for $K=\mathbb{Q}$ \cite[Thm. 2]{Smy86}, and the proof works for any $K$ satisfying the hypotheses of this theorem.

(2) $\implies$ (3): Let $T=\{(x_{i1}, x_{i2},...,x_{in})\}_{i=1}^N$ be a balanced multiset of tuples of size $N$. For $j: 1 \leq j \leq n$, let $v_j = (x_{ij})_{i=1}^N$ be the vector in $K^N$ obtained by taking the $j^{th}$ entry from each tuple in $T$. By definition of balanced multiset, there exist (not necessarily unique) $N \times N$ permutation matrices $X_1,...,X_n$ such that $X_i v_n = v_i$ for all $i$. Thus $(\sum_{i=1}^n a_i X_i) v_n = \sum_{i=1}^n a_i v_i = 0$, so $\sum_{i=1}^n a_i X_i$ has nontrivial kernel.

(3) $\implies$ (2): Reverse the previous argument as follows. Let $v_n$ be any nonzero vector in the kernel of $\sum_{i=1}^n a_i X_i$ and let $v_i:= X_i v_n$ for $1 \leq i \leq n$. Then the coordinates of the vectors $v_1,...,v_n$ give a balanced multiset of tuples as above.
\end{proof}

Our main theorem is the following.

\begin{theorem} \label{smyth-conj-Fqt}
Let $n \geq 3$ be an integer. A coprime tuple $(a_1,...,a_n) \in \mathbb{F}_q[t]^n$ is a Smyth tuple if and only if $a_1, \dots, a_n$ satisfy the absolute value criteria over $\bbF_q(t)$.
\end{theorem}

The forward direction of Theorem \ref{smyth-conj-Fqt} was established by Smyth over $\mathbb{Q}$ in a proof that is valid over a general field $K$. We record this result in Lemma \ref{balanced-implies-criteria}.



\begin{lemma} \label{balanced-implies-criteria} \cite[Cor. 2]{Smy86}
Let $K$ be any field. If $(a_1,...,a_n) \in K^n$ is a Smyth tuple, then $a_1,...,a_n$ satisfy the absolute value criteria over $K$.
\end{lemma}

\section{Proof of Theorem \ref{smyth-conj-Fqt}} \label{proof section} 
We now turn toward the proof of our main result, Theorem \ref{smyth-conj-Fqt}, which will follow essentially as a corollary from Proposition \ref{prop:SmythTupleNumberOfSolutions}. 

Let $V_N \subset \bbF_q[t]$ be the set of polynomials of degree $ < N$.
We will show that, surprisingly, for any coprime $(a_1,..,a_n) \in \mathbb{F}_q[t]^n$ satisfying the absolute value criteria and any $N \geq d= \max \deg_i a_i$, the set of \textit{all} solutions $(x_1,...,x_n) \in V_N^n$ to the equation $\sum_{i=1}^n a_i x_i = 0$ is a balanced set.

\begin{prop}
\label{prop:SmythTupleNumberOfSolutions}
Let $n \geq 3$ be an integer. Let $(a_1, \dots, a_n) \in \bbF_q[t]^n$ of height $d$ be a coprime tuple satisfying the absolute value criteria. Let $N \geq d$ be an integer and let $j$ be an integer so that $1 \leq j \leq n$. Fix $x_j \in V_N$. Then the number of tuples $(x_1, \dots, x_{j-1}, x_{j+1}, \dots x_n)\in V_N^{n-1}$ satisfying $\sum_{i=1}^n a_i x_i = 0$ is $q^{N(n-2)-d}$. In particular, this count does not depend on $j$.
\end{prop}
\begin{proof}
Without loss of generality, we let $j = 1$. By the absolute value criteria, the maximum degree of $a_1, \dots, a_n$ is achieved at least twice. Hence, some $a_i$ with $i \neq 1$ has degree $d$; without loss of generality, assume that $a_n$ does. Let $c=a_1x_1$. Define $$S=\{(x_2, \dots, x_n) \in V_N^{n-1}: c+\sum_{i=2}^n a_i x_i = 0\} $$ Our goal is to compute $\# S$. To do so, we will project onto $\bbF_q[t]/a_n$, so we define $$ \overline{S} = \{ (\overline{x_2}, \dots, \overline{x_{n-1}})\in( \bbF_q[t]/a_n)^{n-2} : \overline{c}+\sum_{i=2}^{n-1}\overline{a_ix_i}=0\}$$ 

Reducing modulo $a_n$ in each coordinate and throwing out the last coordinate gives a surjective $q^{(N-d)(n-2)}$-to-1 map $S \to \overline{S}$; the pre-image of any $(\overline{x}_2,...,\overline{x}_{n-1}) \in \overline{S}$ is $$\left\{\left(x_2 + h_2 a_n,..., x_{n-1} + h_{n-1}a_n, -\left(\frac{c+\sum_{i=2}^{n-1}a_ix_i}{a_n}+\sum_{i=2}^{n-1} a_ih_i\right)\right): h_i \in V_{N-d}\right\},$$ where $x_i$ is the unique polynomial of degree $<d$ equal to $\overline{x}_i$ mod $a_n$. Thus, we have $\#S = q^{(N-d)(n-2)} \#\overline{S}$.

So we want to count the number of solutions $(\overline{x}_2,...\overline{x}_{n-1}) \in (\bbF_q[t]/(a_n))^{n-2}$ to $\overline{c} + \sum_{i=2}^{n-1} \overline{a_ix_i} = 0$. Let $a_n=\prod p_j^{e_j}$ be the prime factorization of $a_n$. Let $R_j:= \bbF_q[t]/(p_j^{e_j})$; by the Chinese Remainder Theorem, it will suffice to count the number of solutions in $R_j$ for each $j$.

Specifically, let $\overline{S}_j = \{(\overline{x_2},..., \overline{x}_{n-1}) \in R_j^{n-2} : \bar{c} + \sum_{i=2}^{n-1} \overline{a_i} \overline{x_i} = 0\}$; then the Chinese Remainder Theorem implies that 

$$\# \overline{S} = \prod_j \# \overline{S}_j.$$

We now compute $\# \overline{S}_j$. Recall that by the absolute value criteria, for all $j$, there are at least two $a_i$ that are not divisible by $p_j$. Of course $a_n$ is divisible by all $p_j$; hence, for all $j$, there is at least one $a_i$, with $1 < i < n$, such that $p_j \nmid a_i$, in which case $\overline{a_i}$ is a unit in $R_j$. Thus, we can write $\overline{x_i}=\frac{\overline{c}+\sum_{k \neq i, k \neq 1} \overline{a_k x_k}}{\overline{a_i}}$, and so any collection of choices of $\overline{x_k}$ for $k \in \{2,...,n-1\} \setminus \{i\}$, will give a unique choice of $\overline{x_i}$. There are $\#R_j = q^{\deg(p_j^{e_j})}$ choices for each $\overline{x}_k$, so $\# \overline{S}_j = q^{\deg(p_j^{e_j})(n-3)}$. Thus, since $\sum_j \deg(p_j^{e_j}) = \deg(a_n) = d$, we have $$ \# S = q^{(N-d)(n-2)} \prod_j q^{\deg(p_j^{e_j})(n-3)}= q^{(N-d)(n-2)} q^{d(n-3)} = q^{N(n-2)-d}, \text{ as desired.}$$
\end{proof}

Now Theorem \ref{smyth-conj-Fqt} follows easily.

\begin{proof}[Proof of Theorem \ref{smyth-conj-Fqt}] 

($\Rightarrow$): This was Lemma \ref{balanced-implies-criteria}. 

($\Leftarrow$): Without loss of generality, we assume $(a_1, \dots, a_n)$ is a coprime tuple in $\bbF_q[t]^n$. Let $T_N$ be the set of all tuples $(x_1,\dots,x_n) \in V_N^n$ satisfying $\sum_{i=1}^n a_i x_i = 0$ and enumerate $T_N = \{(x_{i1},\dots,x_{in})\}_{i=1}^t$ where $t = |T_N|$. In Proposition \ref{prop:SmythTupleNumberOfSolutions} we showed that for every $x \in V_N$ and all $i: 1 \leq i \leq n$, the number of tuples $(x_1,\dots,x_n)$ in $T_N$ with $x_i = x$ is $q^{N(n-2)-d}$. This means that for each $j: 1 \leq j \leq n$, the multiset $\{(x_{ij})\}_{i=1}^t$ is precisely $q^{N(n-2)-d}$ copies of $V_N$. Thus, $T_N$ is a balanced (multi)set of tuples. So by Proposition \ref{equivalent-conditions}, $(a_1,\dots,a_n)$ is a Smyth tuple.
\end{proof}

\begin{remark} \label{1-factor-case}
By setting $N=d= \max_i \deg(a_i)$ and $n=3$ in Proposition \ref{prop:SmythTupleNumberOfSolutions}, we see that if $(a_1,a_2,a_3)$ is a Smyth triple, then $T_d$ is a 1-factor.
\end{remark}

\section{Degree of Galois Conjugates in Linear Galois Relations} \label{deg-of-gal-conj}

In the previous section, we produced balanced multisets for all Smyth tuples, i.e. all tuples for which it is possible to find a balanced multiset. We relied on Proposition \ref{equivalent-conditions} to turn the balanced multisets into linear Galois relations. In this section, we consider the question of how close these linear Galois relations are to ``as small as possible.'' As a first attempt to formulate this question, we might ask for a uniform lower bound $\ell_n(D)$ on the degree of all linear Galois relations whose coefficients are an $n$-tuple of height $D$. However, this is quickly seen to be an uninteresting question, as there exist Smyth tuples of arbitrarily large height that are coefficients of a bounded degree linear Galois relation. (For instance, if $\sum_{i=1}^n \epsilon_i a_i = 0$ where $\epsilon_i \in \{1,-1\}$, then there is a balanced multiset with respect to $(a_1,...,a_n)$ of size $2$.) So instead, we ask for a lower bound $L_n(D)$ on the ``worst'' Smyth $n$-tuple of height $D$.

\begin{definition} \label{L_n(D)}
Let $L^K_n(D)$ be the minimal $L$ so that there is some $(a_1,...,a_n) \in K^n$ of height $D$ such that every linear Galois relation with coefficients $a_1,...,a_n$ has degree at least $L$ 
\end{definition}

In this section the underlying field $K$ will always be $\mathbb{F}_q(t)$, and so we will omit it from the notation. In Section \ref{smyth-over-Q} we will record analogous results about $L^\mathbb{Q}_3(D)$.

In Corollary \ref{ex:BigBalancedSetOfTuples}, we prove that there exist Smyth triples $(a_1,a_2,a_3)$ of height $D$ for which one cannot find a balanced multiset of size less than $|V_D \setminus \{0\}|=q^D-1$. In other words, besides removing zero from the multiset, our construction in Section 4 cannot be uniformly improved upon. 

In Proposition \ref{forward}, we (non-constructively) show that a $3$-term linear Galois relation of degree $d$ with coefficients $(a_1,a_2,a_3)$ implies the existence of a balanced multiset with respect to $(a_1,a_2,a_3)$ of size $d$. Thus we obtain in Corollary \ref{L3(D)-bound} the lower bound $q^D -1 \leq L_3(D)$.

Note that this does not imply that the linear Galois relations arising from our balanced multisets in the $n=3$ case have degree as small as possible; the linear Galois relations arising from a balanced multiset of size $q^D-1$ may have degree as large as $(q^D-1)!$ (see Theorem \ref{Gal-relation-from-1-factor}).

\subsection{A Map of Representations}

We will use the framework of representation theory to prove Proposition \ref{forward}. We now introduce in a general setting the representations we will be working with. 

Let $K$ be a field and let $L/K$ be a Galois extension with Galois group $G$. Let $\gamma \in L$ be any element and let $V_{GalConj}$ be the $K$-vector space spanned by the conjugates of $\gamma$. Then $V_{GalConj}$ is a $G$-invariant subspace of $L$, and so the action of $G$ on $L$ as a $K$-vector space gives an action of $G$ on $V_{GalConj}$. We call this the \textit{Galois conjugate representation}. 

Suppose that $\gamma$ has degree $d$ over $K$. Let $\gamma_1, \dots, \gamma_d$ be the Galois conjugates of $\gamma$. Let $V_{Perm}$ be a $d$-dimensional vector space over $K$ with basis $\alpha_1,...,\alpha_d$ and define an action of $G$ on $V_{Perm}$ by $g \cdot \alpha_i = \alpha_j$ whenever $g(\gamma_i) = \gamma_j$. This is the \textit{permutation representation}. We will write $(V_{Perm}, \rho_{Perm})$ and $(V_{GalConj},\rho_{GalConj})$ to denote the permutation and Galois conjugate representations of $G$, respectively. 

Notice that specializing the formal variables $\alpha_i$ to the Galois conjugates $\gamma_i$ gives a surjective map $\phi$ of $G$-representations from $(V_{Perm}, \rho_{Perm})$ to $(V_{GalConj},\rho_{GalConj})$.




\begin{prop}\label{forward}
Let $K=\bbF_q(t)$. Suppose that $(a_1,...,a_n) \in \mathbb{F}_q[t]^n$ is a coprime tuple and that $\gamma = \gamma_1,...,\gamma_n$ are Galois conjugates over $K$ such that $\sum_{i=1}^n a_i \gamma_i = 0$. Let $m := [K(\gamma): K]$. Then there exists a balanced multiset of tuples with respect to $(a_1,...,a_n)$ of size $m$.
\end{prop}

\begin{proof}
Let $E/K$ be the Galois closure of $K(\gamma_1,...,\gamma_n)$, and write $G:=\Gal(E/K)$.
Let $\gamma= \gamma_1$ and choose $g_1 =1,...,g_n \in G$ so that $g_i\gamma = \gamma_i$. Notice that $\sum_{i=1}^n a_i g_i$ gives a linear transformation on both $V_{Perm}$ and $V_{Gal}$, which we'll call $\psi_{Perm}$ and $\psi_{Gal}$ respectively. Moreover, they commute with the map of $G$-representations $\phi$, i.e. the following diagram commutes.

\begin{center}
\begin{tikzcd}
V_{Perm} \arrow[d, "\psi_{Perm}"]  \arrow[r, "\phi"] & V_{GalConj} \arrow[d, "\psi_{GalConj}"] \\ V_{Perm}  \arrow[r, "\phi"] & V_{GalConj}
\end{tikzcd}
\end{center}

By assumption there is a nonzero $\gamma \in \ker \psi_{GalConj}$ and a simple diagram-chasing argument shows that $\psi_{Perm}$ has nontrivial kernel too. Therefore $\det(\sum_{i=1}^n a_i g_i) = 0$ on $V_{Perm}$, and the construction from Proposition \ref{equivalent-conditions} yields a balanced multiset of tuples of size $m$.



\end{proof}

Smyth also produced balanced multisets of tuples from linear Galois relations in \cite[Thm. 2]{Smy86}. Given a linear Galois relation $\sum_{i=1}^n a_i \gamma_i=0$, Smyth's argument yielded balanced multisets of size the degree of the smallest normal extension $\bbQ(\beta)$ of $\bbQ$ containing $\gamma$ so that $[\bbQ(\beta):\bbQ(\gamma)] \geq n$. Our argument shows the existence of a smaller balanced set, of size $[\bbQ(\gamma):\bbQ]$.

\begin{prop} \label{closed under mult by -a/b}
Let $(a,b,c) \in \mathbb{F}_q[t]^3$ be a coprime triple. Any nonzero balanced set with respect to $(a,b,c)$ (if one exists) has size as least the order of $-a/b$ in $(\mathbb{F}_q[t]/c)^*$.
\end{prop}
\begin{proof}
Let $T = \{(t_{1j}, t_{2j}, t_{3j})\}_{j=1}^m$ be any balanced multiset of triples with respect to $(a,b,c)$. Without loss of generality, suppose that the $t_{1j}$ ($j=1,2,...,m$) are coprime polynomials. Choose $k$ so that $t_{1k}$ is not divisible by $c$. 

For a polynomial $x \in \mathbb{F}_q[t]$, we will use the notation $\overline{x}$ to denote the image of $x$ under the mod-c quotient map $\mathbb{F}_q[t] \to \mathbb{F}_q[t]/c$. From the equation $a t_{1k} + b t_{2k} + ct_{3k} = 0$, we have $\overline{t_{2k}} = \overline{(-a/b)t_{1k}}$. By balancedness, this shows that there is some $j$ so that $\overline{t_{1j}} = \overline{(-a/b) t_{1k}}$.

Iterating this argument, we see that $\{\overline{(-a/b)^nt_{1k}}:n \in \mathbb{Z}^+\} \subset \{\overline{t_{1j}}\}_{j=1}^m$, and the result follows.
\end{proof} 


In the case where $c$ is irreducible and $-a/b$ is a generator for $(\bbF_q[t]/c)^*$, we get the following.

\begin{cor}\label{ex:BigBalancedSetOfTuples}
Let $(a,b,c) \in \bbF_q[t]^3$ be a Smyth triple. Suppose that $c$ is irreducible with $\deg(c)=d \geq \deg(b), \deg(a)$ and also that $(-a/b)$ is a generator for $(\bbF_q[t]/c)^*$. Then the smallest balanced multiset of triples with respect to $(a,b,c)$ has size $q^d-1$.
\end{cor}


 \begin{proof}
 From Proposition \ref{prop:SmythTupleNumberOfSolutions}, setting $N = d$ and $n=3$, and removing the triples $(0,0,0)$, we obtain a balanced multiset (in fact set) of triples of size $q^d-1$. 
 
 On the other hand, Proposition \ref{closed under mult by -a/b} shows that this there cannot be a smaller balanced multiset.
 \end{proof}

\begin{cor} \label{L3(D)-bound}
$L_3(D) \geq q^D-1$.
\end{cor}
\begin{proof}
By Proposition \ref{forward}, it suffices to show that for all $D$, there is some triple $(a,b,c)$ with respect to which every balanced multiset of tuples has size at least $q^D-1$. In other words, we are reduced to constructing a triple satisfying the hypotheses of Corollary $\ref{ex:BigBalancedSetOfTuples}$ with $d=D$.

 Let $b,c$ be distinct irreducible polynomials of degree $D$. Let $g \in (\mathbb{F}_q[t]/c)^*$ be any multiplicative generator and let $a$ be the unique polynomial of degree $\leq D-1$ such that $a = -gb$ (mod c). Then $a,b,c$ are pairwise coprime and their maximum degree is achieved twice, so they are a Smyth triple by Theorem \ref{smyth-conj-Fqt}. Thus they satisfy the hypotheses of Corollary $\ref{ex:BigBalancedSetOfTuples}$.
\end{proof}


\section{Galois Group of Constructed Linear Galois Relations} \label{gal-grp}

In this section we show that the conjugates in our construction for Smyth triples from Section \ref{proof section} can be chosen to have the full symmetric group as their Galois group.


We begin with a lemma.
\begin{lemma}
\label{lem:NormalBasis}
Let $L$ be an $S_d$-extension of $K: = \mathbb{F}_q(t)$. There exist conjugates $\alpha_1, \dots, \alpha_d$ with $\sum_{i=1}^d \alpha_i \neq 0$ such that $L=K(\alpha_1, \dots, \alpha_d)$. 
\end{lemma}

\begin{proof}




Choose $\alpha \in L$ of degree $d$ with conjugates $\alpha=\alpha_1,...,\alpha_d$ such that $L = K(\alpha_1,...,\alpha_d)$. We will show that $\sum_{i=1}^d \alpha_i \neq 0$ for some $1 \leq k \leq d$. Let $p_k=\sum_{i=1}^d x_i^k$ be the $k^{th}$ power sum polynomial. Suppose for contradiction that $p_k(\alpha_1, \dots, \alpha_d)=0$ for all $k \geq 1$. Let $e_i$ be the $i^{th}$ elementary symmetric polynomial and let $f(x)=\sum_{i=0}^d a_i x^i$ be the minimal polynomial of $\alpha_1$. We have that $e_k(\alpha_1, \dots, \alpha_d) = (-1)^k a_{d-k}$ for $1 \leq k \leq d$. 

 
Since we assume $p_i(\alpha_1, \dots, \alpha_d)=0$ for all $1 \leq i \leq d$, the Newton identities give $$(-1)^{k-1}ke_k(\alpha_1, \dots, \alpha_d)=0$$ So, $e_k(\alpha_1, \dots, \alpha_d)=0$ for all $p \nmid k$. Thus, the minimal polynomial of $\alpha_1$ is inseparable, contradicting the fact that $\alpha$ generates a separable extension. So, there is some $k$ for which $\sum_{i=1}^d \alpha_i^k \neq 0$. If $\sum_{i=1}^d \alpha_i = 0$, then $k > 1$ and the $\alpha_i^k$ are not guaranteed to generate $L$. But we can find some nonzero $c \in K$ for which $\alpha_1 + c \alpha_1^k$ and its conjugates generate $L$. Indeed, for all but finitely many $c \in K$, the elements $\alpha_i + c \alpha_i^k$ are pairwise distinct, in which case no nontrivial $K$-automorphisms of $L$ fix the field $K(\alpha_i + c \alpha_i^k)_{i=1}^d$. Since $\sum_{i=1}^d(\alpha_1 + c \alpha_1^k) = c\sum_{i=1}^d\alpha_1^k \neq 0$ for nonzero $c$, we are done.

\end{proof}

\begin{theorem} \label{Gal-relation-from-1-factor}
Let $K = \mathbb{F}_q(t)$ and fix an integer $n \geq 3$. Let $(a_1, \dots, a_n) \in \mathbb{F}_q[t]^n$ be a coprime tuple. Suppose that there is a 1-factor $T = \{(x_{i_1},...,x_{i_n}) \in K^n\}_{i=1}^m$ with respect to $(a_1,...,a_n)$ of size $m$. Then there exist Galois conjugates $\gamma_1, \dots, \gamma_d$ such that $\sum_{i=1}^n a_i \gamma_i=0$ and $\Gal(K(\gamma_1)/K)=S_m$.
\end{theorem}
\begin{proof}
By Lemma \ref{lem:NormalBasis}, we obtain $\alpha_1, \dots, \alpha_d$ such that $K(\alpha_1, \dots, \alpha_d)$ is an $S_d$-extension of $K$ and $\sum_{i=1}^m \alpha_i \neq 0$. Define $\gamma_1:=\sum_{j=1}^m v_{1,j} \alpha_j$. Note that since $S$ is a 1-factor, the entries $v_{1,j}$ $(j=1,2,...,m)$ are distinct. We will write $v_j$ for $v_{1,j}$. If $K(\gamma_1)$ is a proper subextension of $L$, then it must be fixed by a nontrivial element $\sigma$ of $S_m$. So suppose that $$\sum_{j=1}^m v_{j} \alpha_j=\gamma_1=\sigma(\gamma_1)=\sum_{j=1}^m v_{\sigma^{-1}(j)} \alpha_j.$$ Then $\sum_{j=1}^m(v_j-v_{\sigma^{-1}(j)})\alpha_j=0$. By \cite[Thm 13.4.2]{Smy86}, which is presented as a lemma about $\bbQ$ but is valid for $K$ as well, the coefficients $v_j-v_{\sigma^{-1}(j)}$ must all be equal. Then for $c := v_1 - v_{\sigma(1)}$, we have $c\sum_{j=1}^m \alpha_j = 0$. By assumption, $\sum_{j=1}^m \alpha_j \neq 0$, so $c=0$. Therefore $\sigma = 1$, and we have proved that $\gamma_1$ is not contained in a proper subextension of $K(\gamma_1,...,\gamma_d)$.

\end{proof}

\section{A Heuristic for the Existence of a Linear Galois Relation with Prescribed Galois Group} \label{section -- heuristic}
We just saw in Section \ref{gal-grp} that the conjugates appearing in the linear Galois relations we constructed for Smyth triples $(a,b,c) \in \mathbb{F}_q[t]^3$ can be chosen to generate an $S_{q^d}$-extension of $\mathbb{F}_q(t)$ where $d$ is the height of $(a,b,c)$. In this section, we consider which Galois groups we expect to arise in this context.

\begin{definition} \label{Smyth-G-tuple}
If $(a_1,...,a_n)$ is a Smyth tuple, $G \subset S_m$ is a finite group, and there exist Galois conjugates $\gamma_1,...,\gamma_n$ over $K$ such that $\sum_{i=1}^n a_i \gamma_i = 0$ and $\Gal(K(\gamma_1)/K) \cong G$, then we say $(a_1,...,a_n)$ is a Smyth $G$-tuple.
\end{definition}

\begin{question} \label{ques}
Suppose that $(a_1,...,a_n) \in \mathbb{F}_q[t]^n$ is a Smyth tuple. For which $m >0$ and transitive subgroups $G \subset S_m$ is $(a_1,...,a_n)$ a Smyth $G$-tuple? 
\end{question}

Heuristic \ref{heuristic}, which depends only on $|G|$ and $m$, will correctly predict that Smyth tuples $(a_1,...,a_n)$ are always $S_{q^N}$-tuples for sufficiently large $N$. (Recall that Remark \ref{1-factor-case} and Theorem \ref{Gal-relation-from-1-factor} together show that every Smyth triple is an $S_{q^N}$-triple for sufficiently large $N$.)

By Proposition $\ref{equivalent-conditions}$, a coprime $n$-tuple $(a_1,...,a_n) \in \mathbb{F}_q[t]^n$ of height $d$ is a Smyth tuple if and only if there exist permutation matrices $X_1,...,X_n$ such that $\det(\sum_{i=1}^n a_i X_i) = 0$. Our approach so far has been to look for balanced multisets, which is essentially asking, given $(a_1,...,a_n)$, whether or not there exists a vector $v \in \mathbb{F}_q[t]^n$ and permutation matrices $X_1,..X_n$ such that $(\sum_{i=1}^n a_i X_i)v =0$. We now shift our perspective slightly by fixing both the tuple $(a_1,...,a_n)$ and a vector $v$, and asking whether or not there exist such permutation matrices in $G$. Since $v$ will have no repeated entries, the heuristic is suited only to detect 1-factors, as opposed to a general balanced multiset.

In particular, we will fix a vector $v_N \in \mathbb{F}_q[t]^{q^N}$ whose coordinates are all the elements of $V_N$ in some order, each occurring exactly once. We have already seen that for any Smyth triple $(a_1,a_2,a_3)$, there do exist permuation matrices $X_1, X_2, X_3 \in S_{q^N}$ such that $(a_1X_1 + a_2X_2 + a_3X_3)v_N = 0$ (Remark \ref{1-factor-case}). Our heuristic, in the case of triples, suggests an answer to the question of whether this remains true if we insist on choosing $X_1,X_2,X_3$ from some subgroup $G \subset S_{q^N}$. 


\begin{Heuristic} \label{heuristic}
Fix $N$ and enumerate the set of polynomials over $\mathbb{F}_q$ of degree $< N$: $V_N = \{f_1,f_2,...,f_{q^N}\}$, and fix a vector $v = (f_i)_{i=1}^{q^N} \in \mathbb{F}_q[t]^{q^N}$ whose coordinates are all the elements of $V_N$, each occurring exactly once. Let $X_n = I$ be the $q^N \times q^N$ identity matrix. Then choose \textit{random} permutations $X_1,...,X_{n-1} \in G \subset S_{q^N}$ and assume that for each $j: 1 \leq j \leq q^N$, the sum $\sum_{i=1}^n a_i v_{X_i^{-1}(j)}$ takes values in $V_{N+d}$ uniformly and independently at random. 
\end{Heuristic}

\begin{remark}
A shortcoming of Heuristic \ref{heuristic} is that it doesn't see the (necessary) ``local'' conditions -- it would predict that \textit{any} $n$-tuple $(a_1,...,a_n)$, even those which are not Smyth tuples, is a Smyth $G$-tuple for large enough $G$. 
\end{remark}

\begin{remark}
The reason for fixing $X_n = I$ is that if $\sum_{i=1}^n a_i X_i v = 0$, then also $\sum_{i=1}^n a_i (XX_i) v = 0$ for every $X \in G$. This means that we are really concerned with $G$-orbits of $n$-tuples of $G$, and we identify these with $(n-1)$-tuples of elements of $G$ by fixing the last coordinate.
\end{remark}

The result of this assumption is that we have $\sum_{i=1}^n (a_i X_i) v = 0$ with probability $(q^{-(N+d)})^{q^N}$. On the other hand, we have $|G|^{n-1}$ choices for $X_1,...,X_{n-1}$, so, according to this model, the probability of failing to find $X_1,...,X_n$ such that $\sum_{i=1}^n (a_i X_i) v = 0$ is $p_N:= (1-q^{-(N+d)q^N})^{|G|^{(n-1)}}$.

This probability goes to $0$ as $N \to \infty$ for $|G| = \omega_{q,d}(q^{(N+d)q^N/(n-1)})$ (Proposition \ref{l'hopital}). (We are using the standard Landau asymptotic notation here, in which $f(n) = \omega(g(n))$ means that $\lim_{n \to \infty} \frac{g(n)}{f(n)} = 0$.) 

Notably, if we plug in $n=2$, we see that Heuristic \ref{heuristic} correctly does \textit{ not } predict that a random pair $(a_1,a_2)$ is a Smyth pair; indeed, one can check that $(a_1,a_2) \in \mathbb{Z}^2$ is a Smyth pair if and only if $a_1 = \pm a_2$.

\begin{prop} \label{l'hopital}
Write $|G_N| = c_N(q^{(N+d)q^N})^{1/(n-1)}$ and suppose that $\lim_{N \to \infty} c_N = \infty$. Then $\lim_{N \to \infty} p_N = 0$.
\end{prop}

By Stirling's formula, Heuristic \ref{heuristic} suggests an affirmative answer to Question \ref{ques} for $G = S_{q^N}$ when $n \geq 3$, in accordance with the fact that we found linear Galois relations among triples of conjugates which generate $S_{q^N}$-extensions.  It also predicts an affirmative answer for the alternating groups $A_{q^N}$, but predicts negative answers for the cyclic group of order $q^N$ and the dihedral group of order $2q^N$.








\section{Smyth's Conjecture Over $\mathbb{Q}$} \label{smyth-over-Q}

In this brief section we note that some of our results from Section \ref{deg-of-gal-conj} can be easily translated from $\mathbb{F}_q(t)$ over to $\mathbb{Q}$. In particular, we have the following analogues of Proposition \ref{closed under mult by -a/b}, Corollary \ref{ex:BigBalancedSetOfTuples}, and Corollary \ref{L3(D)-bound}.

\begin{prop}
\label{ex:BigBalancedSetOfTuplesQ}
If $(a,b,c) \in \bbZ^3$ then any balanced multiset of triples with respect to $(a,b,c)$ (if one exists) has size at least the order of $-a/b$ in $(\mathbb{Z}/c\mathbb{Z})^*$.
\end{prop}

The proof of this proposition is virtually identical to the proof of Proposition $\ref{closed under mult by -a/b}$. We also obtain analogous corollaries.

\begin{cor}
Let $(a,b,c) \in \mathbb{Z}^3$ be a coprime triple. Suppose that $c$ is prime and that $-a/b$ is a generator for $(\mathbb{Z}/c\mathbb{Z})^*$. Then the smallest balanced multiset of triples with respect to $(a,b,c)$ (if one exists) has size at least $|(\mathbb{Z}/c\mathbb{Z})^*| = c-1$.
\end{cor}

\begin{cor}
Let $p_D$ denote the largest prime at most $e^D$. Assuming Conjecture \ref{Smyth-Conj}, we have $L^\mathbb{Q}_3(D)\geq p_D-1$ for $D \geq 5$, and consequently, $L_3^\mathbb{Q}(D) = \Omega(e^D)$.
\end{cor}
\begin{proof}
Similar to Corollary \ref{L3(D)-bound}, we just have to construct a fitting Smyth triple. Assuming the Smyth Conjecture, for $(a,b,c) \in \bbZ^3$ to be a Smyth triple is the same as for $a,b,c$ to be pairwise coprime and satisfy the triangle inequalities ($a+b\geq c$, $a+c\geq b$, $b+c\geq a$).

Assume that $D \geq 5$ and let $p_D$ be the largest prime at most $e^D$ and let $g \in (\bbZ/p_D\bbZ)^*$ be any multiplicative generator other than $-1 \mod p_D$. 


Take $n$ to be the representative of $\frac{-1}{g+1} \mod p_D$ between $0$ and $p_D-1$. Then, $-\frac{n+1}{n}\equiv g \mod p_D$ and if $n + (n+1) \geq p_D$, then $(n, n+1, p_D)$ satisfy the absolute value criteria over $\mathbb{Q}$. (Since we are assuming Conjecture \ref{Smyth-Conj}, this is equivalent to being a Smyth triple.) On the other hand, if $n + (n+1) < p_D$, then $(p_D-(n+1), p_D-n, p_D)$ satisfies the absolute value criteria. In either case, the inequality now follows from Proposition \ref{ex:BigBalancedSetOfTuplesQ}.


By the prime number theorem, for any $k < 1$, we have $p_N>kN$ for large enough $N$. Therefore $p_D -1 = \Omega(e^D)$.

\end{proof}

\section{Smyth's Conjecture Over Number Fields} \label{num-fields}

Recall that in any field, the absolute value criteria are necessary conditions for being a Smyth tuple (Lemma \ref{balanced-implies-criteria}). We showed in Theorem \ref{smyth-conj-Fqt} that these criteria are sufficient for being a Smyth tuple over $\mathbb{F}_q(t)$, and Smyth conjectured the same over $\mathbb{Q}$ (Conjecture \ref{Smyth-Conj}). 

However, an example presented by David Speyer in a MathOverflow post shows that the absolute value criteria are not sufficient for being a Smyth tuple in a general number field \cite{MO264035}. In particular, the triple $(1,1,\frac{1+\sqrt{-15}}{2})$ satisfies the absolute value criteria, but is not a Smyth triple. Note that this triple achieves equality in the archimedean absolute value inequalities. 

Speyer showed in the same post that for triples of the form $(1,1,a_3)$, if one amends the absolute value criteria to be strict inequalities for the archimedean absolute values, then they become a sufficient condition for being a Smyth triple.
On the other hand, examples such as $(2,3,-5)$ show that we cannot simply amend the archimedean absolute value criteria to be strict inequalities, as $(2,3,-5)$ trivially \textit{is} a Smyth triple. Instead, if some analogue of Smyth's Conjecture is true in number fields, it must be a little more sensitive to the cases in which there is equality in one of the archimedean absolute value criteria. 

In order to formulate what we think the right conjecture is, we define the \textit{strong absolute value criteria over $K$} as follows.

\begin{itemize}
    \item[(1'')] For any archimedean absolute value $| \cdot |$ of $K$, we have $|a_i| <\sum_{j \neq i} |a_j|$ for all $i$.
    \item[(2'')] For any nonarchimedean absolute value $| \cdot |$ of $K$, we have $|a_i| \leq \max_{j \neq i} |a_j|$ for all $i$.
\end{itemize}

The strong absolute value criteria are obtained from the absolute value criteria by making the archimedean inequalities strict.

We are now ready to formulate our generalization of Conjecture \ref{Smyth-Conj}.

\begin{conj} \label{Smyth-Conj-Num-Field}
Let $K$ be a number field and $\cO_K$ its ring of integers. $(a_1,...,a_n) \in \cO_K^n$ is a Smyth tuple if and only if $(a_1,...,a_n)$ satisfy the strong absolute value criteria over $K$ or there exist roots of unity $\omega_1,..., \omega_n$ in some extension of $K$ such that $\sum_{i=1}^n a_i \omega_i = 0$.
\end{conj}

\begin{remark} The $K = \mathbb{Q}$ case of Conjecture \ref{Smyth-Conj-Num-Field} is equivalent to Conjecture \ref{Smyth-Conj}.
\end{remark}

We will show in Proposition \ref{archimedean-equality} that Conjecture \ref{Smyth-Conj-Num-Field} correctly deals with the tuples in which equality is achieved in one of the archimedean absolute value criteria. In particular, if $(a_1,...,a_n)$ is a tuple such that equality in one of the archimedean absolute value criteria, then any tuple in a balanced multiset with respect to $(a_1,...,a_n)$ (if one exists) is a scalar multiple of a tuple of roots of unity.

But first we need two lemmas, the first of which shows that the property of being a Smyth tuple is preserved by multiplying the coordinates by (possibly different) roots of unity.

\begin{lemma} \label{multiply-by-root-of-unity}
Let $(a_1,...,a_n) \in \cO_K^n$. If $\omega_1,...,\omega_n$ are roots of unity in some extension of $K$ and $(a_1,..., a_n)$ is a Smyth tuple, then $(\omega_1 a_1,..., \omega_n a_n)$ is a Smyth tuple in $\cO_{K(\omega_1, \dots, \omega_n)}^n$.
\end{lemma}
\begin{proof}
Without loss of generality we may assume that $\omega_2 = ... = \omega_n = 1$, as we can make the following argument about each coordinate in turn. Denote $\omega := \omega_1$ and $L := K(\omega)$. Suppose that $\omega^m = 1$.

Let $\{(x_{i_1},...,x_{i_n}) \in K^n\}_{i=1}^N$ be a balanced multiset with respect to $(a_1,...,a_n)$. Then $\bigcup_{k=0}^{m-1} \{(\omega^{k-1}x_{i_1},\omega^kx_{i_2},...,\omega^k x_{i_n}) \in L^n\}_{i=1}^{N}$ is a balanced multiset with respect to $(\omega a_1,a_2,...,a_n)$.
\end{proof}

\begin{remark} \normalfont
In particular, Lemma \ref{multiply-by-root-of-unity} shows that if there are roots of unity $\omega_1,...,\omega_n$ such that $\sum_{i=1}^n a_i \omega_i =0$, then $(a_1,...,a_n)$ is a Smyth tuple. Linear relations among roots of unity are a well-studied topic, going back at least to the 1960s. There are several results constraining the prevalence of such relations, indicating that such coefficients represent quite a small subset of Smyth tuples. A survey of some of these results is given in \cite{Zan95}. For instance, when $a_1,...,a_n$ are rational, a result of Mann gives an explicit upper bound depending only on $n$ for the order of the roots of unity $\omega_i$ occurring in a minimal relation $\sum_{i=1}^n a_i \omega_i=0$ \cite{Man65}. (Here minimality means that no nonempty proper sub-sum vanishes, and that the equation is normalized so that $\omega_1 = 1$.)
\end{remark}


\begin{lemma} \label{all-coordinates-1}
Let $(a_1,...,a_n) \in \cO_K^n$. Suppose that there exists an archimedean absolute value $|\cdot |_\nu$ of $K$ and some $i$ for which $|a_i|_\nu = \sum_{j \neq i} |a_j|_\nu$. If there exists a balanced multiset with respect to $(a_1,...,a_n)$, then there exists a balanced multiset $\{(y_{i1},...,y_{in}) \in K^n\}_{i=1}^N$ with respect to $(a_1,...,a_n)$ whose coordinates $y_{ij}$ all satisfy $|y_{ij}|_\nu = 1$.
\end{lemma}
\begin{proof}
Without loss of generality assume that $|a_1|_\nu = \sum_{j > 1} |a_j|_\nu$.
Let $S=\{(x_{i1},...,x_{in}) \in K^n\}_{i=1}^N$ be a balanced multiset with respect to $(a_1,...,a_n)$. Let $X = \{x_{ij}: 1 \leq i \leq N, 1 \leq j \leq n\}$ be the set of all coordinates appearing in $S$. Write $M = \max_{x \in X} |x|_\nu$. Any reference to ``absolute value'' in this proof refers to $|\cdot|_\nu$. 

We claim that if a tuple in $S$ has a coordinate of absolute value $M$, then all coordinates of that tuple have absolute value $M$. To see this, first suppose that $|x_{i_01}|=M$ for some $i_0$. Along with the assumptions that $\sum_{j=1}^n a_j x_{i_0j} = 0$ and $|a_1|_\nu = \sum_{j > 2} |a_j|_\nu$, this implies that $|x_{i_0j}|_\nu = M$ for all $j=1,...,n$. What we've shown so far is that if the first coordinate in a tuple in $S$ has absolute value $M$, then all coordinates in that tuple do. 

But $S$ is balanced, which means that the multiset of first coordinates is the same as the multiset of $j^{th}$ coordinates for every $j=1,2,...,n$. In particular, each of these multisets has the same number of elements of absolute value $M$, with the same multiplicities. Therefore coordinates of absolute value $M$ can only occur in tuples whose first coordinate has absolute value $M$, and the claim is proved.

Thus the tuples whose coordinates have absolute value $M$ form a balanced sub-multiset of $S$, and dividing all of these coordinates by an element of $K$ of absolute value $M$, we obtain the desired balanced multiset.
\end{proof}

\begin{prop} \label{archimedean-equality} Let $(a_1,...,a_n) \in \cO_K^n$. Suppose that there exists an archimedean absolute value $|\cdot |_\nu$ of $K$ and some $i$ for which $|a_i|_\nu = \sum_{j \neq i} |a_j|_\nu$. Then $(a_1,...,a_n)$ is a Smyth tuple if and only if there exist roots of unity $\omega_1,...,\omega_n$ (not necessarily in $K$) such that $\sum_{i=1}^n a_i \omega_i = 0$.
\end{prop}
\begin{proof}
($\Leftarrow$): By assumption, $(\omega_1 a_1,..., \omega_n a_n)$ is a Smyth tuple. The result now follows from Lemma \ref{multiply-by-root-of-unity}.

($\Rightarrow$): Let $\phi: K \hookrightarrow \mathbb{C}$ be an embedding corresponding to the archimedean absolute value $|\cdot|_\nu$ and let $\psi: \overline{\mathbb{Q}} \hookrightarrow \mathbb{C}$ be an embedding of the algebraic closure of $\mathbb{Q}$ which extends $\phi$. We will write $|\cdot|$ for the standard absolute value of complex numbers. Without loss of generality assume that $|\phi(a_1)| = \sum_{j > 1} |\phi(a_j)|$.

By Lemma \ref{all-coordinates-1}, there exists a balanced multiset $S=\{(x_{i1},...,x_{in}) \in K^n\}_{i=1}^N$ with respect to $(a_1,...,a_n)$ such that all $|x_{ij}|_\nu = 1$. By definition of balanced multiset, we have
\begin{equation} \label{tuple-equations}
    \sum_{j=1}^n a_j x_{ij} = 0
\end{equation}

Now (\ref{tuple-equations}) along with $|\phi(a_1)| = \sum_{j > 1} |\phi(a_j)|$ and the assumption that $|\phi(x_{ij})| = 1$ implies that

\begin{equation} \label{arg-equation}
    \arg \phi(a_j x_{ij}) = \pi + \arg \phi(a_1 x_{i1}) (\text {mod }2\pi), \text{ for all } i,j \text{ with } j >1.
\end{equation}

In words, (\ref{arg-equation}) is saying that given a fixed $i$, the $\phi(a_j x_{ij})$ all ``point in the same direction'' for $j > 1$, and $\phi(a_1 x_{i1})$ ``points in the opposite direction.''

The rest of the argument is most easily expressed in polar coordinates. For all $j$, let $\phi(a_j) = r_j \theta_j$ where $r_j \in \mathbb{R}^{\geq 0}$ and $|\theta_j|=1$. Fix any $i_0 \in \{1,2,..,N\}$ and any $j \in \{2,...,n\}$. Then by (\ref{arg-equation}) and the fact that all $|\phi(x_{ij})|=1$, we have $\phi(x_{{i_0}j}) = -\frac{\theta_1}{\theta_j}\phi(x_{i_01})$. 


By balancedness, there is some $i_1$ so that $x_{i_11} = x_{i_0j}$, so repeating the above argument, we get $\phi(x_{i_1j}) = -\frac{\theta_1}{\theta_j}\phi(x_{i_11}) = -\frac{\theta_1}{\theta_j}\phi(x_{i_0j}) = (-\frac{\theta_1}{\theta_j})^2 \phi(x_{i_01})$. Iterating, this argument shows that $(-\frac{\theta_1}{\theta_j})^m \phi(x_{i_01}) \in \{\phi(x): x \in X\}$ for all $m \in \mathbb{Z}$, implying that $-\frac{\theta_1}{\theta_j}$ is a root of unity.

Now let $\omega_1 =1$ and $\omega_j = -\frac{\theta_1}{\theta_j}$ for $j >1$. Dividing the equation (\ref{tuple-equations}) with $i=i_0$ by $x_{i_01}$ and applying $\phi$ to both sides, we have $\sum_{i=1}^n \phi(a_i) \omega_i = 0$. Finally, letting $\rho_i = \psi^{-1}(\omega_i)$, we see that $\psi(\sum_{i=1}^n a_i \rho_i) = \sum_{i=1}^n \phi(a_i) \omega_i = 0$, and hence $\sum_{i=1}^n a_i \rho_i = 0$.
\end{proof}

The above work, along with Lemma \ref{balanced-implies-criteria}, reduces Conjecture \ref{Smyth-Conj-Num-Field} to the following.

\begin{conj} \label{reduced-conj} Let $K$ be a number field and $\cO_K$ its ring of integers. If $(a_1,...,a_n) \in \cO_K^n$ satisfies the strong absolute value criteria, then $(a_1,...,a_n)$ is a Smyth tuple.
\end{conj}
Speyer \cite{MO264035} gives a proof of Conjecture \ref{reduced-conj} in the case where $n=3$ and $a_1=a_2$.

Speyer's argument works for general $n$ and $a_1 = ... = a_{n-1}$ with minimal modification; this result is our final proposition.

\begin{prop} Let $n \geq 3$ be an integer. Let $K$ be a number field and $\cO_K$ its ring of integers.
Let $\alpha \in \cO_K$ so that every archimedean absolute value of $\alpha$ is less than $n-1$. Then $(1,1,\dots, 1, \alpha) \in \cO_K^n$ is a Smyth tuple.
\end{prop}

\begin{proof}
By Lemma \ref{multiply-by-root-of-unity}, it suffices to show $(1,1,\dots, 1, -\alpha)$ is a Smyth tuple. By Proposition \ref{equivalent-conditions}, it suffices for us to show that there are permutation matrices $X_i$ so that $\sum_{i=1}^{n-1}X_i$ has $\alpha$ as an eigenvalue.

We will follow the argument from \cite{MO264035}, starting with a slight generalization of Speyer's Step 1, which we write out in full for the sake of clarity.

Step 1: There is a nonnegative integer matrix $C$ whose rows sum to $n-1$ with eigenvalue $\alpha$.

Consider the lattice $A=\bbZ[\alpha]$ and the vector space $V=A \otimes_\bbZ \bbR$. Since $\alpha$ is an algebraic integer, $A$ is a discrete full sublattice of $V$. We take the norm $\sum_{\nu} |x|_{\nu}^2$, where the sum runs over all archimedean places. Let $c = \max_{\nu} |\alpha|_{\nu}$. By hypothesis, $c<n-1$. Denote by $B_R$ the closed ball of radius $R$ around $0$. 

Let $M$ be large enough so that any ball of radius $M$ around any point in $V$ contains a point in $A$. Take $R$ large enough so that $\frac{c}{n-1}R+(n-2)M<R$. Now, for any $z \in A \cap B_R$, let $z_1 \in A \cap B_R$ be the closest point to $\frac{\alpha z}{n-1}$. Let $z_2=\alpha z - (n-2)z_1$. Now, $$|z_1| \leq |z_1-\frac{\alpha z}{n-1}| + |\frac{\alpha z}{n-1}| \leq M + \frac{c}{n-1}R < R$$ Similarly, $$|z_2|=|\alpha z - (n-2)z_1| \leq |\alpha z - \frac{n-2}{n-1} \alpha z| + (n-2)|\frac{\alpha z}{n-1}-z_1| \leq \frac{c}{n-1}R+(n-2)M < R$$ Thus, for any $z \in A \cap B_R$, we can find $z_1,z_2 \in A \cap B_R$ so that $(n-2)z_1+z_2=z$. Enumerate the elements of $A \cap B_R$ as $z_1, z_2, \dots, z_l$. Then, we can form an $l \times l$ matrix $C$ with the following entries. For the $i$-th row, consider $z_i$. As before, we may write $(n-2)z_j+z_k=z_i$ for some $1 \leq j,k \leq l$. In the $i$-th row, put $n-2$ in the $j$-th column and $1$ in the $k$-th column if $j \neq k$; if $j=k$, put an $n-1$ in the $j^{th}$ column. This matrix $C$ has all row sums $n-1$. By construction, it has $\alpha$ as an eigenvalue with right eigenvector $(z_1,z_2, \dots, z_l)^T$.

The rest of Speyer's argument can now be applied with virtually no modification; using the Perron-Frobenius theorem, one obtains a matrix $D$ from $C$ which is the sum of $n-1$ permutation matrices and still has $\alpha$ as an eigenvalue.
\end{proof}

\printbibliography

@article{Smy86, 
title={Additive and Multiplicative Relations Connecting Conjugate Algebraic Numbers}, DOI={https://doi.org/10.1016/0022-314X(86)90094-6}, journal={Journal of Number Theory}, 
year={1986}, 
author = {C.J. Smyth}, 
pages={243–254}
}

@article{Ber04, 
title={The conjugate dimension of algebraic numbers}, 
volume={55}, 
DOI={10.1093/qmath/hah003}, 
number={3}, 
journal={The Quarterly Journal of Mathematics}, author={Berry, N. and Dubickas, Arturas and Elkies, Noam D. and Poonen, Bjorn and Smyth, C.J.}, 
year={2004}, 
pages={237–252}
}

@article{Man65,
title = {On linear relations between roots of unity},
volume={12},
journal={Mathematika},
author={H.B. Mann},
year={1965},
pages={107-117}
}

@article{Zan95,
title = {Vanishing sums of roots of unity},
volume = {53, 4},
journal = {Rend. Sem. Mat. Univ. Pol. Torino},
author = {U. Zannier}, 
year = {1995},
pages = {487-495}
}

@MISC {MO264035,
    TITLE = {Tell me an algebraic integer that isn't an eigenvalue of the sum of two permutations},
    AUTHOR = {David E Speyer (https://mathoverflow.net/users/297/david-e-speyer)},
    HOWPUBLISHED = {MathOverflow},
    NOTE = {URL:https://mathoverflow.net/q/264035 (version: 2017-03-09)},
    EPRINT = {https://mathoverflow.net/q/264035},
    URL = {https://mathoverflow.net/q/264035}
}

\end{document}